\documentclass[final,1p,times]{elsarticle}

\usepackage{amsmath}
\usepackage{amssymb}
\usepackage{amsfonts}
\usepackage{amsthm}
\usepackage{hyperref}
\usepackage{comment}

\newtheorem{theorem}{Theorem}[section]
\newtheorem{lemma}[theorem]{Lemma}
\newtheorem{corollary}[theorem]{Corollary}

\newtheorem{definition}[theorem]{Definition}
\newtheorem{conjecture}[theorem]{Conjecture}

\begin{document}

\begin{frontmatter}

\title{On Square Roots of Unity Modulo $n$:
Connections
to Fibonacci-like Sequences and Twin Primes}

\author{Srikanth Cherukupally}
\ead{sricheru1214@gmail.com}

\begin{abstract}
\noindent For a positive integer n, let $\mathcal{A}(n)$ denote the set of integers $a$ satisfying $1 \leq a < n$ subject to the simultaneous modular conditions \(n \mid a^2 - 1\) and \(a \mid n^2 - 1\). In this paper, we investigate the structural properties and cardinality of \(\mathcal{A}(n)\).  By establishing a  precise algebraic correspondence between the elements of \(\mathcal{A}(n)\) and the integer evaluations of a family of Fibonacci-like sequences defined by a second-order recurrence, we prove that \(\vert{}\mathcal{A}(n)\vert{} < \log_2 n +2\) for all composite $n > 1$. Trivially, $|\mathcal{A}(n)| = 2$ when $n$ is an odd prime power. Furthermore, we analyze the global distribution of these cardinalities, demonstrating that the asymptotic average of \(\vert{}\mathcal{A}(n)\vert{}\), for $n\leq x$, approaches $2 + o(1)$ as $x \to \infty$.
Guided by the properties of the underlying algebraic mapping and supported by computational verification for all $n < 10^{15}$, we formulate the conjecture that $|\mathcal{A}(n)| \le 3$ uniformly for all $n$. In connection with this conjecture, we show that every integer $N$ that is the product of a pair of twin primes satisfies $|\mathcal{A}(N)| = 3$, providing an explicit family of integers $N$ with $|\mathcal{A}(N)| = 3$.

\end{abstract}


\end{frontmatter}

\section{Introduction}

\noindent Let $\omega(x)$ denote the arithmetic function counting the number of distinct prime factors of an integer $x$, and let $\sigma_0(x)$ denote the divisor function representing the total count of positive divisors of $x$. A classical problem in multiplicative and combinatorial number theory concerns the distribution of the modular square roots of unity. 
For \(n=2^\alpha m\), where \(m\) is odd, the number of square roots of unity modulo \(n\) is given by
\begin{equation*}
R(n)=
\begin{cases}
2^{\omega(m)}, & \alpha=0,\\
2^{\omega(m)}, & \alpha=1,\\
2^{\omega(m)+1}, & \alpha=2,\\
2^{\omega(m)+2}, & \alpha\geq3.
\end{cases}
\end{equation*}
In this paper, we introduce a secondary, simultaneous multiplicative constraint on these residues. Specifically, for an integer $n > 1$, we analyze the properties and cardinality of the restricted arithmetic set
\begin{equation}
\mathcal{A}(n) = \{1 \le a < n : n \mid a^2 - 1 \text{ and } a \mid n^2 - 1\}
\end{equation}
The structural elements of $\mathcal{A}(n)$ reside at the intersection of modular quadratic residues and exact divisor spaces. By definition, any element $a \in \mathcal{A}(n)$ must concurrently act as a modular square root of unity modulo $n$ and as an exact divisor of $n^2 - 1$. This dual condition implies an immediate arithmetic upper bound on the cardinality of the set, namely $|\mathcal{A}(n)| \le \min\{R(n), \sigma_0(n^2 - 1)\}$.

For any $n \geq 3$, the boundary values $a = 1$ and $a = n - 1$ trivially satisfy these criteria, guaranteeing that $|\mathcal{A}(n)| \ge 2$. When $n$ is a prime power $p^{\alpha}$ with
$p\geq 3,\alpha \geq 1$, the system collapses to exactly two solutions. Consequently, non-trivial expansion of $\mathcal{A}(n)$ can occur only when $n$ is divisible by at least two distinct prime factors or $n = 2^3$. Indeed, $n = 8$ is the least positive integer with  $|\mathcal{A}(n)|= 3$.

Because the arithmetic function $\omega(n)$ grows systematically with the number of unique prime factors, one might naturally expect the existence of integers $n$ for which $|\mathcal{A}(n)|$ scales arbitrarily large, particularly when $\omega(n)$ approaches its maximal asymptotic order of $\frac{\log n}{\log \log n}$ \cite{hardy1917}. This structural expectation motivates a rigorous exploration into the true supremum of the size of $\mathcal{A}(n)$. Interestingly, exhaustive heuristic computations executed across every integer up to $n < 10^{15}$ fail to yield a single instance where $|\mathcal{A}(n)| > 3$, revealing a surprisingly rigid structural constraint.

To bridge this gap between empirical observation and analytic proof, we introduce a theoretical framework that connects the non-trivial elements of $\mathcal{A}(n)$ to the integer evaluations of a specific family of Fibonacci-like sequences, which are defined by setting $G_0(x) = 1$, $G_1(x) = x$, and
$G_i(x) = xG_{i-1}(x) - G_{i-2}(x)$ for $i \ge 2$. By tracking the mechanics of these polynomial sequences, we establish two primary distribution theorems. First, we prove a deterministic upper bound establishing that $|\mathcal{A}(n)| < \log_2 n + 2$ for all composite $n > 1$, and for $n$ prime, it is easy to prove that $|\mathcal{A}(n)|=2$. Second, by tracking the density of numbers transitioning through these sparse polynomial tracks over a real interval $x$, we prove that the asymptotic average value of $|\mathcal{A}(n)|$ converges to a limit marginally above 2. Finally, we demonstrate that proving the empirical bound $|\mathcal{A}(n)| \le 3$ is strictly equivalent to establishing a disjointness property: that no two distinct integer chains generated by these polynomials can ever intersect or share a common value. We formalize this equivalence as a conjecture on sequence isolation and outline its algorithmic implications for factoring-free enumeration of all $n \leq x$ with
$|\mathcal{A}(n)| \geq 3$.

\section{Connection to Fibonacci-like sequences}
\noindent Let $G_i(x) \in \mathbb{Z}[x]$ be a polynomial, defined recursively by:
\begin{align}
G_0(x) &= 1, \nonumber \\
G_1(x) &= x, \nonumber \\
G_i(x) &= xG_{i-1}(x) - G_{i-2}(x), \quad \text{for } i \ge 2.
\end{align}
It is critical to contrast these entities with classical Fibonacci polynomials \cite{koshy}, which are typically defined via the addition operator: $F_i(x) = xF_{i-1}(x) + F_{i-2}(x)$. (Note also that $G_i(x)$ corresponds directly to the Chebyshev polynomials of the second kind $U_i(x)$ via the scaling $G_i(x) = U_i(x/2)$ \cite{rivlin}.) The insertion of the subtraction operator profoundly alters the underlying irreducibility properties. While specific classical Fibonacci polynomials are irreducible and yield prime values at integer nodes (e.g., $F_2(4) = 17$, $F_4(2) = 29$), the alternate-sign variants $G_i(x)$ are inherently reducible for all $i \ge 2$. For instance, even degrees satisfy the factorized relation $G_{2r}(x) = G_r(x)^2 - G_{r-1}(x)^2$, ensuring they map exclusively to composite values for integer inputs.

The core structural relevance of these polynomials to our target divisibility set $\mathcal{A}(n)$ is established by the following induction identity.

\begin{lemma}\label{lem:induction}
For any index $i \ge 1$ and any integer base $k \ge 2$, the consecutive evaluations satisfy the mutual relation:
\begin{equation}
G_{i-1}(k) \in \mathcal{A}(G_i(k)).
\end{equation}
\end{lemma}

\begin{proof}
We proceed via mathematical induction on the index $j$. We construct a continuous algebraic identity tracking the fractional quotient of adjacent terms, asserting that:
\begin{equation}\label{eq:fraction_identity}
\frac{G_j(k)^2 + G_{j-1}(k)^2 - 1}{G_j(k)G_{j-1}(k)} = k
\end{equation}
Evaluating the base case at $j=1$, we observe:
\begin{equation}
\frac{G_1(k)^2 + G_0(k)^2 - 1}{G_1(k)G_0(k)} = \frac{k^2 + 1 - 1}{k \cdot 1} = k
\end{equation}
Thus, the base case holds. Assume the induction hypothesis holds for all indices up to $j = i$, such that:
\begin{equation}\label{eq:hypo}
G_i(k)^2 + G_{i-1}(k)^2 - 1 = kG_i(k)G_{i-1}(k)
\end{equation}
We now evaluate the configuration at step $j = i + 1$. Utilizing the recurrence definition to substitute for $G_{i+1}(k)$, we obtain:
\begin{align}
\frac{G_{i+1}(k)^2 + G_i(k)^2 - 1}{G_{i+1}(k)G_i(k)} &= \frac{(kG_i(k) - G_{i-1}(k))^2 + G_i(k)^2 - 1}{G_{i+1}(k)G_i(k)} \nonumber \\
&= \frac{k^2 G_i(k)^2 - 2kG_i(k)G_{i-1}(k) + G_{i-1}(k)^2 + G_i(k)^2 - 1}{G_{i+1}(k)G_i(k)}
\end{align}
Applying the inductive hypothesis from Eq. \eqref{eq:hypo} to substitute for the trailing three terms in the numerator yields:
\begin{align}
\frac{k^2 G_i(k)^2 - 2kG_i(k)G_{i-1}(k) + kG_i(k)G_{i-1}(k)}{G_{i+1}(k)G_i(k)} &= \frac{k^2 G_i(k)^2 - kG_i(k)G_{i-1}(k)}{G_{i+1}(k)G_i(k)} \nonumber \\
&= \frac{kG_i(k) \left( kG_i(k) - G_{i-1}(k) \right)}{G_{i+1}(k)G_i(k)}
\end{align}
Recognizing that the expression enclosed in parentheses corresponds exactly to the definition of $G_{i+1}(k)$, the fraction reduces strictly to:
\begin{equation}
\frac{kG_i(k)G_{i+1}(k)}{G_{i+1}(k)G_i(k)} = k
\end{equation}
Consequently, the identity holds universally for all $j \ge 1$. It follows from Eq. \eqref{eq:fraction_identity} that the product $G_i(k)G_{i-1}(k)$ cleanly divides $(G_i(k)^2 + G_{i-1}(k)^2 - 1)$. Because $\gcd(G_i(k), G_{i-1}(k)) = 1$, we can separate the divisibility statements:
\begin{equation}
G_i(k) \mid G_{i-1}(k)^2 - 1 \quad \text{and} \quad G_{i-1}(k) \mid G_i(k)^2 - 1
\end{equation}
This fulfills the exact constraints of the arithmetic set, concluding the proof.
\end{proof}
\begin{definition}
For a fixed integer base $k \ge 3$, the sequence of integer evaluations $\langle G_i(k) \rangle_{i \ge 2}$ is defined as the recurrent track corresponding to $k$, denoted by the orbit $\mathcal{C}_k$.
\end{definition}
By Lemma \ref{lem:induction}, every integer landing on the track $\mathcal{C}_k$ for $k \ge 3$ automatically possesses at least one non-trivial companion element inside its residue set $\mathcal{A}(n)$. We now prove the absolute converse of this property via a descent method structurally equivalent to the classical Euclidean algorithm.
The descent method establishes a unique integer \(z \ge 3\) for each non-trivial element \(a \in \mathcal{A}(n)\) (\(a \notin \{1, n-1\}\)) such that \(a\) and \(n\) appear consecutively in \(\mathcal{C}_{z}\). Summing across non-trivial elements shows that the number of distinct tracks containing \(n > 2\) equals \(\vert{}\mathcal{A}(n)\vert{} - 2\). Consequently, \(\vert{}\mathcal{A}(n)\vert{} \le 3\) is equivalent to the pairwise disjointness of distinct orbits \(\mathcal{C}_{k}\).
Using the monotonicity and growth bounds of \(G_i(x)\) where \(G_i(k) > (k-1)^i\), one can prove the deterministic upper bound \(\vert{}\mathcal{A}(n)\vert{} < \log_2 n +2\) for all \(n > 1\).
For the asymptotic density, the count \(T_{x}\) of integers \(n \le x\) with \(\vert{}\mathcal{A}(n)\vert{} \ge 3\) is bounded by \(2\sqrt{x}\), ensuring the global average value \(B\) converges to a limit of $2+o(1)$. Our computational tests up to \(n < 10^{15}\) - indexed in the OEIS as sequence A394398
\cite{oeisA394398} and A395504 \cite{oeisA395504} - confirm the rigid constraint \(\vert{}\mathcal{A}(n)\vert{} \le 3\).
\section{Orbits of Recurrent Integer Tracks}

By Lemma \ref{lem:induction}, every integer landing on the track $\mathcal{C}_k$ for $k \ge 3$ automatically possesses at least one non-trivial companion element inside its residue set $\mathcal{A}(n)$. We now prove the absolute converse of this property via a descent method structurally equivalent to the classical Euclidean algorithm of finding the GCD of two numbers.
\begin{lemma}\label{lem:descent}
For each non-trivial element $a \in \mathcal{A}(n)$ satisfying $1 < a < n-1$, there exists a unique integer $z \ge 3$ such that $a$ and $n$ emerge as consecutive elements within the orbit $\mathcal{C}_z$.
\end{lemma}
\begin{proof}
Let $a \in \mathcal{A}(n)$ be a non-trivial element. By definition, the properties $n \mid a^2 - 1$ and $a \mid n^2 - 1$ imply that the sum of their squares shifted by 1 must be perfectly balanced by their product, yielding a unique integer parameter $z$:
\begin{equation}
\frac{n^2 + a^2 - 1}{an} = z
\end{equation}
Because $\sqrt{n} < a < n$, it follows that $z \ge 3$. Isolating the higher-order component, we can express the equation as a quadratic relation:
\begin{equation}
n = za - r
\end{equation}
where $r < a$ represents a residual integer step. Substituting this expression back into the structural balance equation reveals that the reduced pair $(r, a)$ satisfies the identical algebraic trait:
\begin{equation}
\frac{a^2 + r^2 - 1}{ar} = z
\end{equation}
This descent generates a sequence of progressively smaller integer co-prime pairs
$$(n,a), (a,r), (r, r_1), (r_1, r_2), \dots$$
while maintaining the constant parameter $z$. Because the values are strictly decreasing and bounded below by zero and each pair is co-prime, this process terminates exactly at the base pair $(z,1)$, matching the precise steps of the classical Euclidean algorithm. Thus, $a$ and $n$ are uniquely identified as adjacent nodes in $\mathcal{C}_z$.
\end{proof}
\begin{corollary}\label{cor:count}
The total number of distinct recurrent tracks $\mathcal{C}_k$ ($k \ge 3$) in which a target integer $n \geq 8$ appears is exactly equal to $|\mathcal{A}(n)| - 2$.
\end{corollary}
\begin{proof}
By Lemma \ref{lem:descent}, each distinct non-trivial element $a \in \mathcal{A}(n)$ maps the integer $n$ into a single unique orbit track. Summing across all non-trivial elements yields the statement. The lower bound of $8$ is due to the fact that the minimum $n$ across all tracks $\mathcal{C}_k$ ($k \ge 3$) is
$G_2(3) = 8$.
\end{proof}
Corollary \ref{cor:count} establishes a powerful logical bridge: an integer $n$ can achieve a set size $|\mathcal{A}(n)| > 3$ if and only if it simultaneously occupies a position on two or more completely distinct polynomial orbits. Thus, the following statements are perfectly equivalent:
\begin{enumerate}
\item For any integer $n > 1$, the cardinality threshold satisfies $|\mathcal{A}(n)| \le 3$.
\item For any pair of distinct bases $k_1 > k_2 \ge 3$, the orbits $\mathcal{C}_k$ do not intersect: $\mathcal{C}_{k_1} \cap \mathcal{C}_{k_2} = \emptyset$.
\end{enumerate}
Based on these structural mechanics, we formalize this boundary restriction.
\begin{conjecture}\label{conj:main}
For any integer $n > 1$, the size of the simultaneous residue set is strictly bounded by $|\mathcal{A}(n)| \le 3$.
\end{conjecture}
We now deploy the growth trajectories of $G_i(x)$ to prove a general deterministic upper bound, verifying that the set size cannot expand uncontrollably.
\begin{lemma}\label{lem:growth}
The family of alternate-sign polynomials $G_i(x)$ satisfies the following inequalities for all indices $i \ge 1$ and bases $k > 1$:
\begin{enumerate}
\item $G_{i+1}(k) > G_i(k)$
\item $G_i(k+1) > G_i(k)$
\item $G_i(k) > (k-1)^i$
\end{enumerate}
\end{lemma}
\begin{proof}
Property 1 follows directly from the subtraction recurrence structure. Property 2 is verified via induction on the base transitions. Setting a fixed value of $k$ and assuming $G_j(k+1) > G_j(k)$ for all $1 \le j \le i$, the step expansion follows:
\begin{equation}
\label{eq:growth_diff}
G_{i+1}(k+1) - G_{i+1}(k) = k \left[ G_i(k+1) - G_i(k) \right] + \left[ G_i(k+1) - G_{i-1}(k+1) \right] + G_{i-1}(k) > 0
\end{equation}
To prove Property 3, we analyze the quotient structure:
\begin{equation}
\frac{G_i(k)}{G_{i-1}(k)} = k - \frac{G_{i-2}(k)}{G_{i-1}(k)}
\end{equation}
By Property 1, the fractional subtraction term is strictly less than 1. Therefore, $\frac{G_i(k)}{G_{i-1}(k)} > k - 1$. Iterating this inequality down to the base case yields $G_i(k) > (k-1)^i$, completing the proof.
\end{proof}
\begin{corollary}
For any integer \(n > 1\), the cardinality of the simultaneous residue set is bounded by:
\begin{equation}
|\mathcal{A}(n)| < \log_2 n + 2.
\end{equation}
\end{corollary}
\begin{proof}
It is known that $n = G_2(3) = 8$ is minimum across all recurrent tracks $\mathcal{C}_k$, $k\geq 3$.
It is minimum $n$ that has non-trivial elements of $\mathcal{A}(n)$.

Suppose an integer $n\geq 8$ appears in an arbitrary orbit $\mathcal{C}_k$ for a minimal base value $k \ge 3$, such that $n = G_i(k)$, $i\geq 2$. If $n$ simultaneously occupies a position in a higher base track $\mathcal{C}_{k'}$, the monotonicity properties proven in Lemma \ref{lem:growth} dictate that it must land at a compressed index position $2\leq j < i$. Thus, the maximum number of distinct orbits $n$ can occupy is strictly bounded by its index level $i$. Invoking Property 3 from Lemma \ref{lem:growth}:
\begin{equation}
n > (k-1)^i \ge (3-1)^i = 2^i \implies i < \log_2 n
\end{equation}
Since the count of non-trivial elements is bounded by the available orbit capacity $i$, we arrive at $|\mathcal{A}(n)| < \log_2 n+2$, $n\geq 8$, taking into account 2 trivial elements of $\mathcal{A}(n)$.
For $2\leq n<8$, it can be verified explicitly that $|\mathcal{A}(n)| < \log_2 n +2$.
\end{proof}
\section{Asymptotic Density and Heuristic Validation}
\noindent Let $T_x$ count the number of integers $n \le x$ that display non-trivial expansions $|\mathcal{A}(n)| \ge 3$. The total number of separate orbits containing elements less than $x$ is bounded by $\lfloor\sqrt{x}+1\rfloor - 2$. Summing the capacities:
\begin{equation}
T_x \leq \sum_{k=3}^{\lfloor\sqrt{x}\rfloor} N(x, k)
\end{equation}
where $N(x, k)$ is the number of integers not exceeding $x$ in track $\mathcal{C}_k$. Applying our index bounds, $N(x, k) < \log_{k-1} x$.
\begin{eqnarray}
T_x & < & \sum_{k=3}^{\lfloor\sqrt{x}\rfloor} \log_{k-1} x \nonumber \\
    & = & \ln x * \sum_{k=3}^{\lfloor\sqrt{x}\rfloor} \frac{1}{\ln (k-1)} \nonumber \\
    & \approx & \ln x \int_{2}^{\lfloor\sqrt{x}\rfloor} \frac{dk}{\ln k} \nonumber \\
    & \approx & \ln x \big(\textrm{Li}(\sqrt{x}) - \textrm{Li}(2)\big) \nonumber \\
    & = &  \ln x * O\bigg(\frac{\sqrt{x}}{\ln(\sqrt{x})}\bigg) \nonumber \\
    & = & O\big(\sqrt{x}\big).
\end{eqnarray}
In the above calculation, $\textrm{Li}(t)$ is the logarithmic integral function that approximates the number of primes $\leq t$.

We evaluate the global asymptotic average value $B$ across the domain:
\begin{equation}
B = \frac{1}{x-1} \sum_{n=2}^{x} |\mathcal{A}(n)| < \frac{1}{x}\left[ 2(x - T_x) -3 + \sum_{i=3}^{\log_2 x+2} i T_i(x) \right]
\end{equation}
where $T_i(x)$ isolates the count of integers where $|\mathcal{A}(n)| = i$. Given $\sum T_i(x) = T_x$, the average expression simplifies to the inequality:
\begin{equation}
B = 2 + O\bigg(\frac{\log_2 x}{\sqrt{x}}\bigg)
\end{equation}
Because the error term decays rapidly as $x \to \infty$, the average value converges asymptotically to  $2 + o\big(1\big)$. This distribution behavior confirms that non-trivial elements are exceedingly sparse throughout the number line.

\section{Remarks}
The union of all orbits $\mathcal{C}_k$, $k\geq 3$ contains integers $n$ such that
$|\mathcal{A}(n)| \geq 3$. If Conjecture \ref{conj:main} is true, we will have
$|\mathcal{A}(n)| = 3$. However, we lack rigorous arguments for proving the conjecture, and only have
empirical evidence for all $n < 10^{15}$ supporting its validity. We make some remarks about twin prime connection and orbit disjointness as first elementary steps towards solving the conjecture.
\subsection*{Twin prime pairs connection}
For $N = p(p+2)$, $\mathcal{A}(N) = \{1,p+1,N-1\}$ if both $p$ and $p+2$ are primes. Such $N$
appears as the second element in the recurrent track $\mathcal{C}_{k}$ for $k = p+1$. Since $|\mathcal{A}(N)| = 3$, $N$ cannot appear as the $i^{th}$
term in any other recurrent track, for $i\geq 3$.
In other words, the subsequence (of the defined sequence), $\langle G_i(k) \rangle_{i,k \ge 3}$,
does not contain twin prime products. It is important to highlight that twin prime products form
a subset of $\{n: |\mathcal{A}(n)| \geq 3\}$ that is proved to contain $N$ with  $|\mathcal{A}(N)| = 3$. If the Twin Prime Conjecture is proved to be true, then it infers
the existence of infinitely many $N$
with $|\mathcal{A}(N)| = 3$.

\subsection*{Orbit disjointness at initial indices}
Proving Conjecture \ref{conj:main} is equivalent to showing
$G_i(k) \not = G_j(l)$ for any $i,j\geq 2 \,\textrm{and}\, k>l\geq 3$. As a subproblem, let us fix $i=2$, and $j=3$.
Suppose $G_2(k) = G_3(l)$ for some $k,l$, then $k^2 - 1 = l^3 - 2l$, which implies
$(k,l)$ is a rational point on the Elliptic curve defined by the equation
$y^2 = x^3 - 2x + 1$. However, there are only 3 rational points on the curve other than the point at infinity: $\{(1,0), (0,1), (0,-1)\}$, contradicting the existence of $k,l\geq 3$. This shows that Orbits $C_k$ and $C_l$ do not intersect at  $2^{nd}$ and $3^{rd}$
positions for any $k>l\geq 3$.

\subsection*{Factoring-free empirical data generation} Computing non-trivial solutions $a\in \mathcal{A}(n)$ requires either factoring $n$ to explicitly compute first square roots $a$ of unity
$\pmod n$ and then verify if $a$ divides $n^2-1$
or factoring $n^2 - 1$ to first compute $a$ and then verify if $n$ divides $a^2 - 1$. The problem of integer factorization is hard in general as $n$ grows
and takes time sub-exponential in $\ln(n)$ for some $n$ (i.e., hard instances of
factoring) by the state-of-the-art Number Field Sieve technique \cite{nfs}. So, for a given $x$, computing all $n\leq x$ such that $|\mathcal{A}(n)|\geq 3$ with the factoring approach is computationally infeasible. However, we have
produced the empirical data on $|\mathcal{A}(n)|$ for all $n < 10^{15}$, just by iteratively
producing elements of each orbit $\mathcal{C}_k$ for $3\leq k<\sqrt{10^{15}+1}$, thus
avoiding explicit factorization of integers $n$, $n^2 - 1$ and $a^2-1$.

\begin{section}*{Acknowledgements}
The author wishes to thank the On-Line Encyclopedia of Integer Sequences (OEIS) community, where the structural sequences corresponding to this research have been indexed as A394398 and A395504. The author wishes to thank Sarosh Adenwalla (Department of Computer Science,
University of Liverpool) for his encouraging comments on the work, based on the Arxiv version of the paper.
\end{section}


\begin{thebibliography}{99}


\bibitem{hardy1917}
G. H. Hardy and S. Ramanujan, \emph{The Normal Number of Prime Factors of a Number $n$}, Quarterly Journal of Mathematics, 48 (1917), 76--92.
\bibitem{hardy1979}
G. H. Hardy and E. M. Wright, \emph{An Introduction to the Theory of Numbers}, 5th ed., Oxford University Press, Oxford, U.K., 1979.
\bibitem{koshy}
T. Koshy, \emph{Fibonacci and Lucas Numbers with Applications}, Vol. 2, John Wiley \& Sons, 2019.

\bibitem{nfs}
A.K. Lenstra, H.W. Lenstra, \emph{Development of the Number Field Sieve}, Springer, 1993.

\bibitem{rivlin}
T. J. Rivlin, \emph{Chebyshev Polynomials: From Approximation Theory to Algebra and Number Theory}, 2nd ed., John Wiley \& Sons, 1990.

\bibitem{oeisA394398}
OEIS Foundation Inc., \emph{Sequence A394398: Number of integers $k$, $1 \le k < n$, such that $n$ divides $k^2-1$ and $k$ divides $n^2-1$}, The On-Line Encyclopedia of Integer Sequences, \url{oeis.org}.
\bibitem{oeisA395504}
OEIS Foundation Inc., \emph{Sequence A395504: Integers $m$ such that there are more than 2 integers $k$, $1 \le k < m$, such that $m$ divides $k^2-1$ and $k$ divides $m^2-1$}, The On-Line Encyclopedia of Integer Sequences, \url{oeis.org}.

\end{thebibliography}
\end{document}